\documentclass[12pt,reqno]{article}

\usepackage[usenames]{color}
\usepackage{amssymb}
\usepackage{graphicx}
\usepackage{amscd}

\usepackage[colorlinks=true,
linkcolor=webgreen,
filecolor=webbrown,
citecolor=webgreen]{hyperref}

\definecolor{webgreen}{rgb}{0,.5,0}
\definecolor{webbrown}{rgb}{.6,0,0}

\usepackage{color}
\usepackage{fullpage}
\usepackage{float}

\usepackage{graphics,amsmath,amssymb}
\usepackage{amsthm}
\usepackage{amsfonts}
\usepackage{latexsym}
\usepackage{epsf}

\newcommand{\seqnum}[1]{\href{http://www.research.att.com/cgi-bin/access.cgi/as/~njas/sequences/eisA.cgi?Anum=#1}{\underline{#1}}}

\begin{document}


\begin{center}
\vskip 1cm{\LARGE\bf Hadamard Products and Tilings } \vskip 1cm
\large Jong Hyun Kim \\
Department of Mathematics\\
Brandeis University\\
Waltham,  MA 02454-9110 \\
USA \\
\href{mailto:jhkim@brandeis.edu}{\tt jhkim@brandeis.edu} \\
\end{center}

\vskip .2 in
\begin{abstract}
Shapiro gave a combinatorial proof of a bilinear generating function for Chebyshev
    polynomials equivalent to the formula
\[ \frac{1}{1-ax-x^2}\ast \frac{1}{1-bx-x^2}
                =  \frac{1-x^2}{1-abx-(2+a^2+b^2)x^2
                -abx^3+x^4},  \]
    where $*$ denotes the Hadamard product. In a similar way, by considering tilings of a $2\times n$ rectangle
    with $1\times1$ and $1\times 2$ bricks in the top row, and $1\times1$ and $1\times n$ bricks in the bottom row,
    we find an explicit formula for the
    Hadamard product
    \[\frac{1}{1-ax-x^2}\ast \frac{x^m}{1-bx-x^n}.\]
\end{abstract}

\newtheorem{theorem}{Theorem}[section]
\newtheorem{lemma}{Lemma}[section]

\catcode `!=11

\newdimen\squaresize
\newdimen\thickness
\newdimen\Thickness
\newdimen\ll! \newdimen \uu! \newdimen\dd! \newdimen \rr! \newdimen \temp!

\def\sq!#1#2#3#4#5{%
\ll!=#1 \uu!=#2 \dd!=#3 \rr!=#4
\setbox0=\hbox{%
 \temp!=\squaresize\advance\temp! by .5\uu!
 \rlap{\kern -.5\ll!
 \vbox{\hrule height \temp! width#1 depth .5\dd!}}%
%
 \temp!=\squaresize\advance\temp! by -.5\uu!
 \rlap{\raise\temp!
 \vbox{\hrule height #2 width \squaresize}}%
%
 \rlap{\raise -.5\dd!
 \vbox{\hrule height #3 width \squaresize}}%
%
 \temp!=\squaresize\advance\temp! by .5\uu!
 \rlap{\kern \squaresize \kern-.5\rr!
 \vbox{\hrule height \temp! width#4 depth .5\dd!}}%
%
 \rlap{\kern .5\squaresize\raise .5\squaresize
 \vbox to 0pt{\vss\hbox to 0pt{\hss $#5$\hss}\vss}}%
}
 \ht0=0pt \dp0=0pt \box0
}

\def\vsq!#1#2#3#4#5\endvsq!{\vbox to \squaresize{\hrule width\squaresize height 0pt%
\vss\sq!{#1}{#2}{#3}{#4}{#5}}}

\newdimen \LL! \newdimen \UU! \newdimen \DD! \newdimen \RR!

\def\vvsq!{\futurelet\next\vvvsq!}
\def\vvvsq!{\relax
  \ifx     \next l\LL!=\Thickness \let\continue!=\skipnexttoken!
  \else\ifx\next u\UU!=\Thickness \let\continue!=\skipnexttoken!
  \else\ifx\next d\DD!=\Thickness \let\continue!=\skipnexttoken!
  \else\ifx\next r\RR!=\Thickness \let\continue!=\skipnexttoken!
  \else\ifx\next P\let\continue!=\place!
  \else\def\continue!{\vsq!\LL!\UU!\DD!\RR!}%
  \fi\fi\fi\fi\fi
  \continue!}

\def\skipnexttoken!#1{\vvsq!}

\def\place! P#1#2#3{%
\rlap{\kern.5\squaresize\temp!=.5\squaresize\kern#1\temp!
  \temp!=\squaresize \advance\temp! by #2\squaresize \temp!=.5\temp!
  \raise\temp!\vbox to 0pt{\vss\hbox to 0pt{\hss$#3$\hss}\vss}}\vvsq!}

\def\Young#1{\LL!=\thickness \UU!=\thickness \DD! = \thickness \RR! = \thickness
\vbox{\smallskip\offinterlineskip \halign{&\vvsq! ## \endvsq!\cr
#1}}}

\def\blank{\omit\hskip\squaresize}
\catcode `!=12

\section{Introduction}
The Fibonacci numbers (\seqnum{A000045}) are defined by $F_0=0$,
$F_1=1$, and for $n\geq2$, $F_n=F_{n-1}+F_{n-2}$. It is convenient
to write $f_n$ for $F_{n+1}$ so that $f_n$ is the number of ways
to tile a $1\times n$ strip with $1\times 1$ square bricks and
$1\times 2$ rectangular bricks \cite{benjamin-quinn}. The number
of tilings of a $1\times n$ strip with $k$ bricks is the
coefficient of $x^n$ in $(x+x^2)^k$, so we know that the
generating function \cite{fosta-han} for Fibonacci numbers is
    \begin{equation} \notag
        \sum_{n=0}^\infty f_n x^n = \frac{1}{1-x-x^2}.
    \end{equation}
We now define the polynomial $f_n(a)$  by
        \begin{equation} \label{eq:11}
            \frac{1}{1-ax-x^2}=\sum_{n=0}^{\infty}f_n(a)x^n.
        \end{equation}
Then we have that $f_n(1)=f_n$ and $f_n(a)$ can be interpreted as
the sum of the weights of tilings of a $1\times n$ strip with
$1\times1$ square bricks weighted by $a$ and $1\times2$
rectangular bricks weighted by $1$. By applying the geometric
series and binomial series to $(1-ax-x^2)^{-1}$ we obtain that
(\seqnum{A011973})
        \begin{equation} \label{eq:12}
        f_n(a)=\sum_{k=0}^{\left\lfloor \frac{n}{2}
        \right\rfloor}{\binom{n-k}{k}}a^{n-2k}.
                        \end{equation}
So, $f_0(a)=1$, $f_1(a)=a$, $f_2(a)=1+a^2$, $f_3(a)=2a+a^3$,
$f_4(a)=1+3a^2+a^4$,
etc.\\

        Louis W. Shapiro \cite{shapiro} gave a combinatorial proof of a bilinear generating function for Chebyshev
    polynomials equivalent to
        \begin{equation} \label{eq:13}
        \sum_{n=0}^\infty f_n(a) f_n(b)x^{2n}=\frac{1-x^4}{1-abx^2-(2+a^2+b^2)x^4-abx^6+x^8} .
                        \end{equation}

The Hadamard product $G\ast H$ of the power series
$G(x)=\sum_{k\geq0} g(k)x^k$ and $H(x)=\sum_{k\geq 0}h(k)x^k$ is
defined by
\[ G\ast H=\sum_{k\geq 0}g(k)h(k)x^k.\]
If $G(x)$ and $H(x)$ are rational power series, then so is the
Hadamard product $G\ast H$ \cite[p.~207]{EC1}.

Using the notation of Hadamard product, we can rewrite
(\ref{eq:13}) as
        \begin{equation} \label{eq:14}
        \frac{1}{1-ax-x^2}\ast \frac{1}{1-bx-x^2}
                =  \frac{1-x^2}{1-abx-(2+a^2+b^2)x^2
                -abx^3+x^4}.
                        \end{equation}
In this paper, I will review Shapiro's proof of (\ref{eq:14}), and
extend this approach to find an explicit formula for the Hadamard
product
    \[\frac{1}{1-ax-x^2}\ast \frac{x^m}{1-bx-x^n}.\]

The MacMahon operator $\Omega_{\geq}$ is defined on formal Laurent series
by
        \begin{equation} \label{eq:014}
         \Omega_{\geq} \sum_{n = -\infty }^{\infty}  a_n x^n 
         =  \sum_{n = 0}^{\infty}  a_n  . 
                               \end{equation}
We assume that the above sum (\ref{eq:014}) converges in an appropriate sense.

We can express the MacMahon operator $\Omega_{\geq}$ in terms of
Hadamard products. Let $G(x)=\sum_{i \geq0} g(i) x^i$ and
$H(x)=\sum_{j \geq 0}h(j) x^j$ be power series. Then
\[ G(x)H(x^{-1})=\sum_{i \geq 0} \sum_{j \geq
    0 }g(i)h(j) x^{i- j} . \]
Since $ \sum_{i\geq 0}(\sum_{j \leq i}h(j))x^i = H( x)/( 1-x)$, we
have that
        \begin{equation} \label{eq:15}
    \Omega_{\geq} G(x)H(x^{-1})=\sum_{i =0}^\infty \sum_{j =0
    }^i g(i)h(j) = \Big( G(x)\ast \frac{H( x )}{1-
    x} \Big)
    \Big|_{x = 1}.
                        \end{equation}
In \cite{Han}, G.-N. Han used computer algebra to show that
        \begin{equation*}
         \Omega_{\geq} \frac{1}{(1-zx -zx^2)(1-y/ x -y/ x^2)}
         =\frac{1+z^2y}{(1-2z)(1-3zy-z^2y-zy^2)}.
         \end{equation*}
We will derive Han's result from our formula for $1/(1-ax-x^2) *
1/(1-bx-x^3)$.

\section{Hadamard products}
Now we review Shapiro's \cite{shapiro} proof of a formula for the
Hadamard product
        \begin{equation} \label{eq:26}
            \frac{1}{1-ax-x^2}\ast
            \frac{1}{1-bx-x^2}=\sum_{k=0}^{\infty}f_k(a)f_k(b)x^k.
        \end{equation}
We consider (\ref{eq:26}) as counting pairs of tilings. The
coefficient $f_k(a)f_k(b)$ of $x^k$ counts tilings of a $2\times
k$ rectangle with $1\times1$ square bricks weighted by $a$ and
$1\times2$ rectangular bricks weighted by $1$ in the top row, and
$1\times1$ square bricks weighted by $b$ and $1\times 2$
rectangular bricks weighted by $1$ in the bottom row, as in the
following figure:
        \squaresize = 16pt \thickness = 1pt \Thickness = 0pt
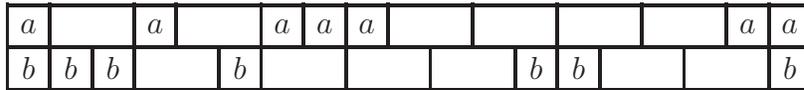
        \[\Young{ a & r & l& a & r & l & a & a & a & r & l & r & l & r & l &r &lr & a & a   \cr
        b& b &b & r &  l & b &lr & lr & r & l & r & l & b & b & r  & l
        &r &l  & b \cr}\]
The vertical line segments passing from top to bottom serve to
factor these tilings into tilings of smaller length. For example,
the following figure shows the factorization of the above figure.
    \squaresize = 14pt \thickness = 1pt \Thickness = 0pt
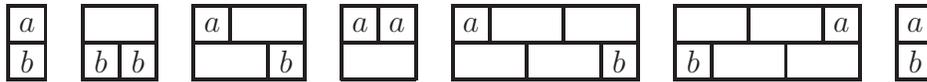
        \[\Young{ a & ud & r & l& ud & a & r & l & ud & a & a & ud & a & r & l & r & l & ud & r & l &r  &l & a & ud & a   \cr
                 b&ud & b &b & ud & r &  l & b & ud & lr & lr & ud & r & l & r & l & b & ud & b& r  & l
        &r &l  & ud & b \cr}\]
Let's define a prime block to be a tiling that cannot be factored
any further without cutting it through the middle of some brick.
So these prime blocks can be classified as follows:

\bigskip
\noindent The prime block of length $1$:
        \squaresize = 16pt \thickness = 1pt \Thickness = 0pt
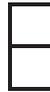
        \[\Young{    \cr
        \cr}\]
The prime blocks of length $2$:
        \[\Young{ r&l  & udrl & udrl & &  & udrl & udrl & r & l \cr
        r&l & udrl , & udrl  &  r & l  & udrl ,& udrl &  &  \cr}\]
The prime blocks of length $2k+1\geq 3$, together with the result
of interchanging the two rows:
        \[\Young{ r&l & r&l &  udr & udrl \cdots & lud & r&l & \cr
        r& r&l & r & l & udr &  rudl \cdots & lud & r & l \cr}\]
The prime blocks of length $2k\geq 4$, together with the result of
interchanging the two rows:
        \[\Young{ r&l & r&l &  udr & udrl \cdots & lud & r&l \cr
        r& r&l & r & l & udr &  rudl \cdots & lud & \cr}\]
Thus the generating function $P_2(x)$ for the weighted prime
blocks of the Hadamard product is
        \begin{align}
        P_2(x) &= abx+(1+a^2+b^2)x^2 +\sum_{k = 1}^{\infty}
        2ab x^{2k+1}+\sum_{k = 2}^{\infty}
        (a^2+b^2)x^{2k} \nonumber \\
        &= abx+(1+a^2+b^2)x^2+\frac{2ab
        x^3}{1-x^2}+\frac{(a^2+b^2)x^4}{1-x^2} \nonumber\\
        &=
        \frac{abx+(1+a^2+b^2)x^2+abx^3-x^4}{1-x^2}.
        \nonumber
        \end{align}
Since any tiling can be factored uniquely as a sequence of prime
blocks \cite[p.~1027--1030]{HC2}, we have $(1-ax-x^2)^{-1} \ast
(1-bx-x^2)^{-1}=1/(1-P_2(x))$. So we obtain the following explicit
formula:
        \begin{equation}\label{eq:27}
                \frac{1}{1-ax-x^2}\ast \frac{1}{1-bx-x^2}
                =  \frac{1-x^2}{1-abx-(2+a^2+b^2)x^2
                -abx^3+x^4},
        \end{equation}
        which is equivalent to Shapiro's result.
Letting $a=1$ and $b=1$ in equation (\ref{eq:27}), we
have \cite[p.~251]{EC1} (\seqnum{A007598})
   \begin{equation*}
        \sum_{n=0}^\infty f_n^2 x^{n} =
         \frac{1-x}{1-2x-2x^2+x^3}.
    \end{equation*}

As noted by Shapiro \cite{shapiro}, (\ref{eq:27}) can be written
as an identity for Chebyshev polynomials. The Chebyshev
polynomials of the second kind $U_n(a)$ $(n\geq0)$
(\seqnum{A093614}) can be defined by the generating function
    \begin{equation*}
         \frac{1}{1-2az+z^2} = \sum_{n=0}^{\infty}U_n(a)z^n.
    \end{equation*}
By substituting $-2ai$ for $a$ and $iz$ for $x$ in equation
(\ref{eq:11}) we have the relation $U_n(a)=i^nf_n(-2ai)$, and from
the identity (\ref{eq:12}) we have
   \begin{equation*}
    U_n(a)=\sum_{k=0}^{\left\lfloor \frac{n}{2}
\right\rfloor}{\binom{n-k}{k}}(-1)^k(2a)^{n-2k}.
   \end{equation*}

By replacing $a$, $b$, and $x$ with $-2ai$, $-2bi$, and $-z$
respectively in equation (\ref{eq:27}) we can obtain the
Chebyshev polynomial identity
\begin{equation*} 
    \begin{split}
        \sum_{n=0}^\infty U_n(a)U_n(b) z^{n}=
         \frac{1-z^2}{1-4abz-(2-4a^2-4b^2)z^2-4abz^3+z^4}.
         \end{split}
    \end{equation*}

We now want to prove an identity which we will use later.
\begin{lemma}\label{eq:51} For $m\geq -1$ and $n\geq -1$,
\begin{align}
f_{m}(a)f_{n+1}(a)-f_{m+1}(a)f_{n}(a)&=(-1)^{\min(m,n+1)}f_{|m-n|-1}(a),\label{eq:28}
\end{align}
where $f_{-1}(a)=0$.
\end{lemma}

\begin{proof}
Fix $m>n>0$ and let $A$ be the set of tilings of a $1\times m$
strip and a $1\times (n+1)$ strip with $1\times1$ square bricks
weighted by $a$ and $1\times2$ rectangular bricks weighted by $1$.
Then there are $f_{m}(a)f_{n+1}(a)$ weighted tilings in $A$.
Similarly, let $B$ be the set of tilings of a $1\times (m+1)$
strip and a $1\times n$ strip with $1\times1$ square bricks
weighted by $a$ and $1\times2$ rectangular bricks weighted by $1$.
Then there are $f_{m+1}(a)f_{n}(a)$ weighted tilings in $B$. Now
we will find a bijection from a subset $A$ to the set $B$ if $n$
is odd, a bijection from a subset of $B$ to the set $A$ if $n$ is
even that proves (\ref{eq:28}). Let's consider a tiling in $A$
drawn in two rows so that the top row is a $1\times m$ strip and
the bottom row is a $1\times (n+1)$ strip indented $m-n$ spaces,
as follows:
        \squaresize = 16pt \thickness = 1pt \Thickness = 0pt
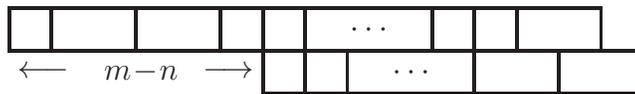
        \[\Young{  & r & l& r & l &  &   & r &lr \cdots &l  &  &   &  r & l     \cr
         ludr\hspace{10pt} \longleftarrow& lrud  &ludr  & lrud \!\!\!\!\!\! m\!-\!n & ludr & lrud  \longrightarrow \hspace{10pt}&    &lr &  r  & lru \cdots
         &lr  & r & l & r & l  \cr}\]
 Let's
find the rightmost vertical line segment, if there is one, that
passes through both strips without cutting through the middle of
some brick. We call the part of the tiling to the right of this
line the tail of the tiling. In the following figure the tail is
separated.
        \[\Young{  & r & l& r & l &  &   & r &lr \cdots &l  &   &udlr &   &  r & l     \cr
         ludr & lrud  &ludr & lrud  & ludr & lrud   &    &lr  &  r  & lru \cdots
          & lr & udr & r & l & r & l    \cr}\]
Switching the two rows of the tail of this tiling produces the
          following tiling in $B$:
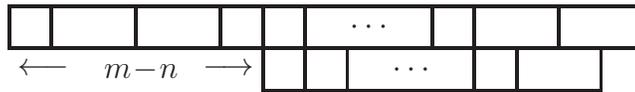
        \[\Young{  & r & l& r & l &  &   & r &lr \cdots &l  &   & r & l & r & l    \cr
         ludr \hspace{10pt}  \longleftarrow& lrud  &ludr  & lrud \!\!\!\!\!\! m\!-\!n & ludr & lrud  \longrightarrow \hspace{10pt} &  &lr &  r  & lru \cdots
         & lr  &   & r & l      \cr}\]
where the top row is a $1\times (m+1)$ strip and the bottom row is
a $1\times n$ strip indented $m-n$ spaces.

When $n$ is odd, this tail switching pairs up every element of $A$
with every element of $B$ except for the tilings in $A$ of the
form:
        \squaresize = 17pt \thickness = 1pt \Thickness = 0pt
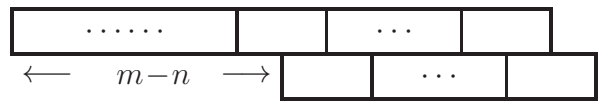
        \[\Young{r & rl & rl \cdots \cdots&  rl  & l & r & l  & r &lr \cdots &l  & r & l  \cr
         ludr \hspace{10pt}  \longleftarrow& lrud  &ludr  & lrud \!\!\!\!\!\! m\!-\!n & ludr & lrud  \longrightarrow \hspace{10pt} &   r &l &  r  & lru \cdots
         &lr  & r & l  \cr}\]
         where the top row is a $1\times m$ strip, the bottom row
         is a
         $1\times(n+1)$ strip indented $m-n$ spaces, $\,\raise -4pt \vbox{\Young{ r & rl & rl \cdots \cdots&  rl  & l \cr}}\,\,$
         represents any strip of length
           $m-n-1$ tiled with $1\times1$ square bricks and $1\times2$
            rectangular bricks, and every other brick is $1\times 2$.
In this case, tail switching cannot be applied to the tiling. So
there are $f_{m-n-1}(a)$ weighted tilings in the set $A$ which
cannot matched with those in the set $B$ by tail switching.
Therefore we have
$f_{m}(a)f_{n+1}(a)-f_{m+1}(a)f_{n}(a)=f_{m-n-1}(a)$.

When $n$ is even, tail switching pairs up every element of $A$
with every element of $B$ except for the tilings in $B$ of the
form:
        \[\Young{ r & rl & rl \cdots \cdots&  rl  & l & r & l  & r &lr \cdots &l  & r & l  &  r & l  \cr
         ludr\hspace{10pt} \longleftarrow& lrud  &ludr  & lrud \!\!\!\!\!\! m\!-\!n & ludr & lrud  \longrightarrow \hspace{10pt}&   r &l &  r  & lru \cdots
         &lr  & r & l  \cr}\]
         where the top row is a $1\times (m+1)$ strip, the bottom row
         is a
         $1\times n$ strip indented $m-n$ spaces, $\,\raise -4pt \vbox{\Young{ r & rl & rl \cdots \cdots&  rl  & l \cr}}\,\,$
         represents any strip of length
           $m-n-1$ tiled with $1\times1$ square bricks and $1\times2$
            rectangular bricks, and every other brick is $1\times 2$.
In this case, tail switching cannot be applied to the tiling. So
there are $f_{m-n-1}(a)$ weighted tilings in the set $B$ which
cannot matched with those in the set $A$ by tail switching.
Therefore we have
$f_{m}(a)f_{n+1}(a)-f_{m+1}(a)f_{n}(a)=-f_{m-n-1}(a)$.

In the case $m>n=0$, the definition of the tail must be modified
slightly. We leave the details to the reader. Now suppose $n>m$.
Let $D(m,n)=f_{m}(a)f_{n+1}(a)-f_{m+1}(a)f_{n}(a)$. Then
$D(m,n)=-D(n,m)=(-1)^mf_{n-m-1}(a)$. This is equivalent to the
desired formula. In the other cases in which $m$ or $n$ is $-1$ or
$m=n$, we can easily see that equation (\ref{eq:28}) is true
because $f_{-1}(a)=0$.
\end{proof}

A special case of the identity (\ref{eq:28}) for $m=n+1$ and $a=1$
is Cassini's Fibonacci identity
$f_{n+1}^2-f_{n+2}f_{n}=(-1)^{n+1}$ which was proved in the same
way in \cite[p.~8]{benjamin-quinn} and \cite{WZ}.

Next, we can use this combinatorial method to obtain an explicit
formula for Hadamard product $(1-ax-x^2)^{-1}\ast
(1-bx-x^n)^{-1}$.
\begin{theorem}\label{eq:61}
       The Hadamard product
            \[\frac{1}{1-ax-x^2}\ast \frac{1}{1-bx-x^n}\]
        is equal to
\[ \frac{1-f_{n-2}x^n}{1-abx-b^2x^2-(f_n+f_{n-2})x^n-(2bf_{n-1}-abf_{n-2})x^{n+1}+(-1)^nx^{2n}}\]
where $f_n$ represents $f_n(a)$, $f_{-1}=0$, and $n\geq 2$.
\end{theorem}

\begin{proof}
We now consider the Hadamard product
        \begin{equation} \label{eq:29}
        \frac{1}{1-ax-x^2}\ast \frac{1}{1-bx-x^n}
        \end{equation}
as counting tilings, using $1\times n$ rectangular bricks instead
of $1\times 2$ rectangular bricks in the bottom row. In this
setting a prime block cannot have a $1\times1$ square brick in the
bottom row anywhere except at the
beginning or end. The possible prime blocks can be classified as follows:\\

\noindent The prime block of length $1$:
        \squaresize = 9pt \thickness = 1pt \Thickness = 0pt
        \[\Young{  \cr
         \cr}\]
The prime blocks of length $n$:
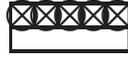
        \[\Young{  \bigotimes & \bigotimes &\bigotimes &\bigotimes& \bigotimes  \cr
        r&lr & lr&lr &l \cr}\]
The prime blocks of length $nk$ $(k\geq2)$:
        \[\Young{  \bigotimes & \bigotimes &\bigotimes &\bigotimes& r&l& \bigotimes& \bigotimes&\bigotimes& r&l& \bigotimes& \bigotimes&\bigotimes & rud  & lrud \cdots &  lud & r&l& \bigotimes& \bigotimes & \bigotimes & r&l& \bigotimes& \bigotimes& \bigotimes & \bigotimes \cr
        r&lr & lr&lr &l& r&lr & lr&lr &l& r&lr & lr&lr &l & udr  & udrl \cdots  &udlr
        & r & lr & lr & lr  & l & lr & lr & lr &lr & l
        \cr}\]
The prime blocks of length $nk+1$ $(k\geq1)$:
        \[\Young{  \bigotimes & \bigotimes &\bigotimes &\bigotimes& r&l& \bigotimes& \bigotimes&\bigotimes& r&l& \bigotimes& \bigotimes&\bigotimes & rud  & lrud \cdots &  lud & r&l& \bigotimes& \bigotimes & \bigotimes & r&l& \bigotimes& \bigotimes& \bigotimes & r & l \cr
        r &lr & lr&lr &l& r&lr & lr&lr &l& r&lr & lr&lr &l & udr  & udrl \cdots  &udlr
        & r & lr & lr & lr  & l & lr & lr & lr &lr & l &
        \cr}\]
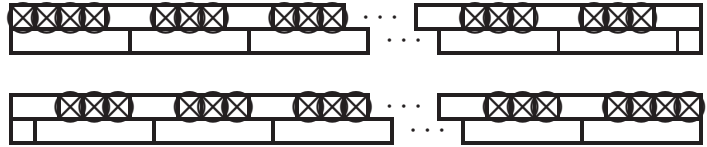
        \[\Young{  r & l & \bigotimes &\bigotimes &\bigotimes& r&l& \bigotimes& \bigotimes&\bigotimes& r&l& \bigotimes& \bigotimes&\bigotimes & rud  & lrud \cdots &  lud & r&l& \bigotimes& \bigotimes & \bigotimes & r&l& \bigotimes& \bigotimes& \bigotimes & \bigotimes \cr
        & r&lr & lr&lr &l& r&lr & lr&lr &l& r&lr & lr&lr &l & udr  & udrl \cdots  &udlr
        & r & lr & lr & lr  & l & lr & lr & lr &lr & l
        \cr}\]
The prime blocks of length $nk+2$ $(k\geq0)$:
        \[\Young{  r & l & \bigotimes &\bigotimes &\bigotimes& r&l& \bigotimes& \bigotimes&\bigotimes& r&l& \bigotimes& \bigotimes&\bigotimes & rud  & lrud \cdots &  lud & r&l& \bigotimes& \bigotimes & \bigotimes & r&l& \bigotimes& \bigotimes& \bigotimes & r & l \cr
        & r&lr & lr&lr &l& r&lr & lr&lr &l& r&lr & lr&lr &l & udr  & udrl \cdots  &udlr
        & r & lr & lr & lr  & l & lr & lr & lr &lr & l &
        \cr}\]
    where $\Young{\bigotimes &\bigotimes &  \bigotimes \cr}$, $\Young{\bigotimes
    & \bigotimes & \bigotimes&\bigotimes  \cr}$, and $\Young{\bigotimes &\bigotimes &  \bigotimes &
    \bigotimes&\bigotimes  \cr}$ represent any strips of length $n-2$, $n-1$, and
    $n$ respectively tiled with $1\times1$ square bricks and $1\times2$
    rectangular bricks.

Thus the generating function $P_n(x)$ for the weighted prime
blocks of the Hadamard product (\ref{eq:29}) is
        \begin{align}
        P_n(x) &= abx+f_n(a)x^n+\sum_{k=1}^{\infty}
        2bf_{n-1}(a)f_{n-2}(a)^{k-1} x^{nk+1}\nonumber \\
        &\qquad + \sum_{k=0}^{\infty} b^2 f_{n-2}(a)^k
        x^{nk+2}  + \sum_{k= 2}^{\infty} f_{n-1}(a)^2 f_{n-2}(a)^{k-2}x^{nk} \nonumber\\
        &=
        abx+f_n(a)x^n+\frac{2bf_{n-1}(a)x^{n+1}}{1-f_{n-2}(a)x^n}+\frac{b^2
        x^2}{1-f_{n-2}(a)x^n}
        +\frac{f_{n-1}(a)^2x^{2n}}{1-f_{n-2}(a)x^n} \nonumber\\
            &= \frac{ abx+b^2x^2+f_n(a)x^n+\big(2bf_{n-1}(a)-abf_{n-2}(a)\big)x^{n+1}
            +(-1)^{n-1} x^{2n}}{1-f_{n-2}(a)x^n}\nonumber
        \end{align}
        where we have used the identity $f_{n-1}(a)^2-f_n(a)f_{n-2}(a)=(-1)^{n-1}$ obtained by substituting $n-1$ for $m$
        and $n-2$ for $n$ in the identity (\ref{eq:28}) of Lemma \ref{eq:51}.
Since any tiling can be factored uniquely as a sequence of prime
blocks, we obtain that $(1-ax-x^2)^{-1}  *
(1-bx-x^n)^{-1}=1/(1-P_n(x))$. This is equivalent to the desired
formula.
\end{proof}
Note that Theorem \ref{eq:61} also holds for $n=1$. The
polynomials $f_n(a)+f_{n-2}(a)$ in Theorem \ref{eq:61} are Lucas
polynomials (\seqnum{A114525}).

Using (\ref{eq:15}) we can prove Han's result
        \begin{equation} \label{eq:30}
         \Omega_{\geq} \frac{1}{(1-z x -zx^2)(1-y/ x -y/ x^2)}
         =\frac{1+z^2y}{(1-2z)(1-3zy-z^2y-zy^2)},
         \end{equation}
by computing the Hadamard product $1/(1-zx-zx^2)
* 1/((1-yx-yx^2)(1-x))$ and then setting
$x=1$. Substituting $z^{1/2}$ for $a$ and $z^{1/2} x$ for $x$ in
$1/(1-ax-x^2)$ gives $ 1/(1-zx -zx^2) $, and substituting
$(y+1)/(-y)^{1/3}$ for $b$ and $(-y)^{1/3} x$ for $x$ in
$1/(1-bx-x^3)$ gives $ 1 /( 1-(y+1) x +y x^3) $. Making these
substitutions in Theorem \ref{eq:61} and using the fact that if
$f(x) * g(x)=h(x)$ then $f(\alpha x)
* g(\beta x)= h(\alpha\beta x)$, we have that the Hadamard product
\[ \frac{1}{1-zx -zx^2}* \frac{1}{1-(y+1) x +y x^3} \]
is equal to
\[ \frac{1+z^2y  x^3 }{1-(zy+z)x-z(y+1)^2x^2 +(3z^2y+z^3y)x^3+(z^3y+2z^2y)(y+1)x^4-z^3y^2x^6}.\]
Then setting $x=1$ gives Han's result (\ref{eq:30}), which he
proved in a more complicated way.

We now modify the above setting to obtain a formula for the
Hadamard product $1/(1-ax-x^2)\ast x^m/(1-bx-x^2)$.
\begin{theorem} The Hadamard product
 \[  \frac{1}{1-ax-x^2}\ast \frac{x^m}{1-bx-x^2} \]
is equal to
        \begin{equation*}
                \frac{f_m(a)x^m+bf_{m-1}(a)x^{m+1}-f_{m-2}(a)x^{m+2}}{1-abx-(2+a^2+b^2)x^2-abx^3+x^4}
                \end{equation*}
                where $f_{-1}(a)=0$, and $f_{-2}(a)=1$.
\end{theorem}

\begin{proof} When $m=0$, the formula reduces to (\ref{eq:27}). When $m\geq 1$, we consider the Hadamard product
                \begin{equation*} 
                \frac{1}{1-ax-x^2}\ast \frac{x^m}{1-bx-x^2}
                \end{equation*}
as counting tilings. We modify the tilings of a $2\times k$
rectangle so that the bottom row starts with a $1\times m$
rectangular brick to account for the factor $x^m$ in
$x^m/(1-bx-x^2)$. In this setting the first block where the bottom
row starts with a $1\times m$ rectangular brick will be different
from all the others, but the following blocks can be built up from
a sequence of prime blocks which are exactly the same as the prime
blocks in the Hadamard product $(1-ax-x^2)^{-1}\ast
(1-bx-x^2)^{-1}$. The
first blocks can be classified as follows:\\

\noindent The first blocks of length $m$:
        \squaresize = 10pt \thickness = 1pt \Thickness = 0pt
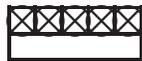
        \[\Young{  \bigotimes & \bigotimes &\bigotimes &\bigotimes& \bigotimes  \cr
        r&lr & lr&lr &l \cr}\]
The first blocks of length $m+2k$ $(k\geq1)$:
        \[\Young{  \bigotimes & \bigotimes &\bigotimes &\bigotimes& r&l & r&l &  udr & udrl \cdots & lud & r&l  \cr
        r&lr & lr&lr &l & r& lr& r  & l & udr &  rudl \cdots & lud &  \cr}\]
The first blocks of length $m+2k+1$ $(k\geq0)$:
        \[\Young{  \bigotimes & \bigotimes &\bigotimes &\bigotimes& r&l & r&l &  udr & udrl \cdots & lud & lr&lr & \cr
        r&lr & lr&lr &l&   r& lr& r  & l & udr &  rudl \cdots & lud & lr & l  \cr}\]
    where $\Young{\bigotimes &\bigotimes &  \bigotimes & \bigotimes \cr}$ and $\Young{\bigotimes &\bigotimes &  \bigotimes &
    \bigotimes&\bigotimes  \cr}$ represent any strips of length $m-1$ and
    $m$ respectively tiled with $1\times1$ square bricks and $1\times2$
    rectangular bricks.
So the generating function $Q_{m,2}(x)$ where $m\geq 2$ for
weighted first blocks is
\begin{align}
Q_{m,2}(x) &= f_m(a)x^m+\sum_{k= 1}^{\infty}
a f_{m-1}(a) x^{m+2k} +\sum_{k=0}^{\infty}b f_{m-1}(a)x^{m+2k+1}  \nonumber\\
&= f_m(a)x^m + \frac{af_{m-1}(a)x^{m+2}}{1-x^2}
+\frac{bf_{m-1}(a)x^{m+1}}{1-x^2} \nonumber \\
&= \frac{f_m(a)x^m+bf_{m-1}(a)x^{m+1}-f_{m-2}(a)x^{m+2}}{1-x^2}.
\nonumber
\end{align}
Since any tiling can be factored uniquely as a first block
followed by a sequence of prime blocks, we have $1/(1-ax-x^2) \ast
x^m/(1-bx-x^2)=Q_{m,2}(x)/(1-P_2(x))$ where $m\geq1$. This is
equivalent to the desired formula.
\end{proof}

Now we can generalize the previous theorem by computing an
explicit formula for the Hadamard product $1/(1-ax-x^2)\ast
x^m/(1-bx-x^n)$.

\begin{theorem}\label{eq:63} The Hadamard product
\[\frac{1}{1-ax-x^2}\ast \frac{x^m}{1-bx-x^n}\]
is equal to
\[ \frac{f_mx^m+bf_{m-1}x^{m+1}+(-1)^{\min(m-1,n-1)}f_{|m-n+1|-1}x^{m+n}}{1-abx-b^2x^2-(f_n+f_{n-2})x^n-(2bf_{n-1}-abf_{n-2})x^{n+1}+(-1)^nx^{2n}}
\]
where $f_n$ represents $f_n(a)$, $f_{-1}=0$, $m\geq 0$, and $n\geq
2$.
\end{theorem}

\begin{proof} When $m=0$, Theorem \ref{eq:63} reduces to Theorem \ref{eq:61}.
Let's consider the Hadamard product
        \begin{equation*}
                \frac{1}{1-ax-x^2}\ast \frac{x^m}{1-bx-x^n}
                \end{equation*}
as counting pairs of tilings where $m\geq1$ and $ n\geq 2$. We
slightly modify the above tiling by using $1\times n$ rectangular
bricks instead of $1\times 2$ rectangular bricks in the bottom
row. In this setting, the first block where the bottom row starts
with a $1\times m$ rectangular brick will be different from all
the others, but the following blocks can be built up from a
sequence of prime blocks whick are exactly the same as the prime
blocks in the Hadamard product $(1-ax-x^2)^{-1}\ast
(1-bx-x^n)^{-1}$. The
possible first blocks can be classified as follows:\\

\noindent The first blocks of length $m$:
        \squaresize = 10pt \thickness = 1pt \Thickness = 0pt
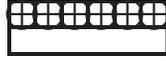
        \[\Young{  \bigoplus& \bigoplus & \bigoplus &\bigoplus &\bigoplus& \bigoplus  \cr
        r&lr & lr&lr& lr &l \cr}\]
The first blocks of length $m+nk$ $(k\geq1)$:
        \[\Young{ \bigoplus & \bigoplus & \bigoplus &\bigoplus &\bigoplus& r&l& \bigotimes& \bigotimes&\bigotimes& r&l& \bigotimes& \bigotimes&\bigotimes & rud  & lrud \cdots &  lud & r&l& \bigotimes& \bigotimes & \bigotimes & r&l& \bigotimes& \bigotimes& \bigotimes & \bigotimes \cr
        r&lr & lr &  lr&lr &l& r&lr & lr&lr &l& r&lr & lr&lr &l & udr  & udrl \cdots  &udlr
        & r & lr &lr & lr  & l & lr & lr & lr &lr & l
        \cr}\]
The first blocks of length $m+nk+1$ $(k\geq0)$:
        \[\Young{ \bigoplus&  \bigoplus & \bigoplus &\bigoplus &\bigoplus& r&l& \bigotimes& \bigotimes&\bigotimes& r&l& \bigotimes& \bigotimes&\bigotimes & rud  & lrud \cdots &  lud & r&l& \bigotimes& \bigotimes & \bigotimes & r&l& \bigotimes& \bigotimes& \bigotimes & r & l \cr
        r &lr & lr&lr&lr &l& r&lr & lr&lr &l& r&lr & lr&lr &l & udr  & udrl \cdots  &udlr
        & r & lr & lr & lr  & l & lr & lr & lr &lr & l &
        \cr}\]
where $\Young{\bigotimes &\bigotimes &  \bigotimes \cr}$,
$\Young{\bigotimes
    & \bigotimes & \bigotimes&\bigotimes  \cr}$, $\Young{\bigoplus &\bigoplus &  \bigoplus &
    \bigoplus&\bigoplus  \cr}$, $\Young{\bigoplus &\bigoplus &  \bigoplus &
    \bigoplus&\bigoplus & \bigoplus \cr}$ represent any strips of length $n-2$, $n-1$, $m-1$ and
    $m$ respectively tiled with $1\times1$ square bricks and $1\times2$
    rectangular bricks.

So the generating function $Q_{m,n}(x)$ where $m\geq 1$ and $n\geq
2$ for the weighted first blocks is

\begin{align}
    Q_{m,n}(x) &= f_m(a)x^m+\sum_{k= 1}^{\infty}
    f_{m-1}(a)f_{n-1}(a)f_{n-2}(a)^{k-1} x^{m+nk} \nonumber\\
    &\qquad +\sum_{k=0}^{\infty} bf_{m-1}(a)f_{n-2}(a)^k x^{m+nk+1}  \nonumber\\
    &= f_m(a)x^m + \frac{f_{m-1}(a)f_{n-1}(a)x^{m+n}}{1-f_{n-2}(a)x^n}
    +\frac{bf_{m-1}(a)x^{m+1}}{1-f_{n-2}(a)x^n} \nonumber \\
    &=
    \frac{f_m(a)x^m+bf_{m-1}(a)x^{m+1}+\Big(f_{m-1}(a)f_{n-1}(a)-f_m(a)f_{n-2}(a)\Big)x^{m+n}}{1-f_{n-2}(a)x^n}
    \nonumber \\
    &=
    \frac{f_m(a)x^m+bf_{m-1}(a)x^{m+1}+(-1)^{\min(m-1,n-1)}f_{|m-n+1|-1}(a)x^{m+n}}{1-f_{n-2}(a)x^n}
    \nonumber
\end{align}
where we use the identity
$f_{m-1}(a)f_{n-1}(a)-f_m(a)f_{n-2}(a)=(-1)^{\min(m-1,n-1)}f_{|m-n+1|-1}(a)$
obtained by substituting $m-1$ for $m$
        and $n-2$ for $n$ in the identity (\ref{eq:28}) of Lemma \ref{eq:51}.
Since any tiling can be factored uniquely as a first block
followed by a sequence of prime blocks, we have that $1/(1-ax-x^2)
\ast x^m/(1-bx-x^n) = Q_{m,n}(x)/(1-P_n(x))$ where $m\geq 1$ and
$n\geq 2$. This is equivalent to the desired formula.
\end{proof}
Note that Theorem \ref{eq:63} also holds for $n=1$. In Theorem
\ref{eq:63}, there are some special cases: when $b=0$, we have
that
\[ \sum_{k \geq 0 }f_{m+nk}(a)x^{m+nk}=
\frac{f_m(a)x^m+(-1)^{\min(m-1,n-1)}f_{|m-n+1|-1}(a)x^{m+n}}{1-(f_n(a)+f_{n-2}(a))x^n+(-1)^nx^{2n}}.\]In
particular, when $n=1$ and $b=0$, we have that
\[ \sum_{k\geq
m}f_k(a)x^k=\frac{f_m(a)x^m+(-1)^{\min(m-1,0)}f_{m-1}(a)x^{m+1}}{1-ax-x^2}.\]

Using a similar method, we can also compute an explicit formula
for the Hadamard product $x^m/(1-ax-x^2)*1/(1-x^n)$.

\begin{theorem}\label{eq:64} For positive integer $m\geq 1$ and $n \geq2$, the Hadamard product
\[\frac{x^m}{1-ax-x^2}\ast \frac{1}{1-x^n}\]
is equal to
\[ \frac{f_{n-r}(a)x^{(q+1)n}+(-1)^{n-r-1}f_{|r-1|-1}(a)x^{(q+2)n}}{1-(f_n(a)+f_{n-2}(a))x^n+(-1)^nx^{2n}}
\]
if $m=qn+r$ for some positive integers $q$ and $r$ with $0< r < n$, and is equal to
\[ \frac{x^m-f_{n-2}(a)x^{m+n}}{1-(f_n(a)+f_{n-2}(a))x^n+(-1)^nx^{2n}}
\]
if $m=qn$ for some positive integer $q$.
\end{theorem}

\begin{proof}
Let's consider the Hadamard product
        \begin{equation*}
                \frac{x^m}{1-ax-x^2}\ast \frac{1}{1-x^n}
                \end{equation*}
as counting pairs of tilings. We modify the tilings of a $2\times
k$ rectangle in the proof of Theorem \ref{eq:63} so that the top
row starts with a $1\times m$ rectangular brick to account for the
factor $x^m$ in $x^m/(1-ax-x^2)$. In this setting the first block
where the top row starts with a $1\times m$ rectangular brick will
be different from all the others, but the following blocks can be
built up from a sequence of prime blocks whick are exactly the
same as the prime blocks in the Hadamard product $1/(1-ax-x^2)\ast
1/(1-x^n)$. The
possible first blocks can be classified as follows:\\

\noindent The first blocks of length $(q+1)n$:
        \squaresize = 10pt \thickness = 1pt \Thickness = 0pt
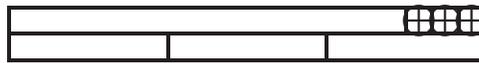
        \[\Young{ r  & rl &rl & rl & lr & rl & rl & rl
         & rl &rl & rl & rl &rl  &rl  &l &\bigoplus & \bigoplus & \bigoplus\cr
        r&lr & lr &  rl& lr &l & r&lr &lr &lr & lr &  lr &r&lr &lr &lr & lr &  l    \cr}\]
The first blocks of length $(q+k)n$ $(k\geq 2)$:
        \[\Young{  r  & rl &rl & rl & lr & rl & rl & rl
         & rl &rl & rl & rl &rl &rl &l &\bigoplus &  \bigoplus  & rud  & lrud \cdots &  lud &
          r&l& \bigotimes& \bigotimes & \bigotimes & \bigotimes & r&l& \bigotimes& \bigotimes& \bigotimes & \bigotimes & \bigotimes \cr
        r&lr & lr &  rl& lr &l & r&lr &lr &lr & lr &  lr &r&lr &lr &lr & lr &  l & udr& udrl \cdots&udl
        & r & lr & lr &lr & lr  & l & lr & lr & lr & lr &lr & l
        \cr}\]
where $\Young{\bigoplus &\bigoplus \cr}$, $\Young{\bigoplus
    & \bigoplus & \bigoplus   \cr}$, $\Young{\bigotimes &\bigotimes &  \bigotimes &
    \bigotimes  \cr}$, $\Young{\bigotimes &\bigotimes &  \bigotimes &
    \bigotimes&\bigotimes\cr}$ represent any strips of length $n-r-1$, $n-r$, $n-2$ and
    $n-1$ respectively tiled with $1\times1$ square bricks and $1\times2$
    rectangular bricks.

So the generating function $R_{m,n}(x)$ where $m=qn+r$ and $0< r
<n $ for the weighted first blocks is

\begin{align}
    R_{m,n}(x) &= f_{n-r}(a)x^{(q+1)n}+\sum_{k= 2}^{\infty}
    f_{n-1-r}(a)f_{n-1}(a)f_{n-2}(a)^{k-2} x^{(q+k)n} \nonumber\\
    &=
    f_{n-r}(a)x^{(q+1)n} + \frac{f_{n-1-r}(a)f_{n-1}(a)x^{(q+2)n}}{1-f_{n-2}(a)x^n}   \nonumber \\
    &= \frac{f_{n-r}(a)x^{(q+1)n}+  \Big(
    f_{n-1}(a)f_{n-r-1}(a)-f_{n-2}(a)f_{n-r}(a)\Big)x^{(q+2)n}}{1-f_{n-2}(a)x^n
    }   \nonumber \\
    &=
    \frac{f_{n-r}(a)x^{(q+1)n}+(-1)^{\min(n-r-1)}f_{|r-1|-1}(a)x^{(q+2)n}}{1-f_{n-2}(a)x^n} \nonumber
\end{align}
where we use the identity
$f_{n-1}(a)f_{n-r-1}(a)-f_{n-2}(a)f_{n-r}(a)=(-1)^{\min(n-r-1,n-1)}f_{|r-1|-1}(a)$
obtained by substituting $n-2$ for $m$
        and $n-r-1$ for $n$ in the identity (\ref{eq:28})
        of Lemma \ref{eq:51}. Since any tiling can be factored
uniquely as a first block followed by a sequence of prime blocks,
we have that $x^m/(1-ax-x^2) \ast 1/(1-x^n) =
R_{m,n}(x)/(1-P_n(x))$ where $b=0$. This is equivalent to the
desired formula.

When $m=qn$, we have only one
first block of length $m$:\\
        \squaresize = 10pt \thickness = 1pt \Thickness = 0pt
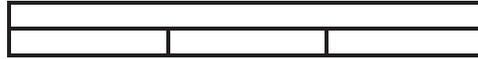
        \[\Young{ r  & rl &rl & rl & lr & rl & rl & rl
         & rl &rl & rl & rl &rl  &rl  &rl &rl &rl &l   \cr
        r&lr & lr &  rl& lr &l & r&lr &lr &lr & lr &  lr &r&lr &lr &lr & lr &  l    \cr}\]
So the generating function $R_{m,n}(x)$ for the weighted first
blocks is $R_{m,n}(x) = x^m $. Therefore we have that
$x^m/(1-ax-x^2) * 1/(1-x^n)=x^m/(1-P_n(x))$ where $b=0$. This is
equivalent to the desired formula
\end{proof}
Note that Theorem \ref{eq:64} also holds for $n=1$.



\noindent \emph{Keywords: } Hadamard product, Fibonacci polynomials.








\begin{thebibliography}{99}

\bibitem{benjamin-quinn} A. T. Benjamin and J. J. Quinn, {\it Proofs That Really Count}: {\it The Art of Combinatorial Proof}, Mathematical Association of America, Washington, DC, 2003.

\bibitem{fosta-han} D. Foata and G.-N. Han, Nombres de Fibonacci et
polyn\^omes orthogonaux, in M. Morelli and M. Tangheroni, eds.,
{\it Leonardo Fibonacci: Il Tempo, Le
Opere, L'Eredit\`a Scientifica}, 
Pacini, Rome, 1994, pp.\ 179--208.

\bibitem{Han} G.-N. Han, A general algorithm for the MacMahon omega
operator, {\it Ann. Combin.} {\bf 7} (2003) 467--480.

\bibitem{HC2} I. M. Gessel
and R. P. Stanley, Algebraic enumeration, in 
R. L. Graham, M. Gr\"{o}tschel, and L.  Lov\'{a}sz, eds.,
{\it Handbook of
Combinatorics}, Vol.\ 2, Elsevier and MIT Press, 1995, pp.\ 1021--1062.

\bibitem{shapiro} L. Shapiro, A
combinatorial proof of a Chebyshev polynomial identity, {\it
Discrete Math}. {\bf 34} (1981) 203--206.

\bibitem{EC1} R. P. Stanley, {\it
Enumerative Combinatorics}, Vol. 1, Cambridge University Press,
2002.

\bibitem{WZ} M. Werman and D. Zeilberger, A bijective proof of Cassini's
Fibonacci identity, {\it Discrete Math}. {\bf 58} (1986) 109.

\end{thebibliography}
\end{document}